\documentclass[12pt]{amsart}
\usepackage{graphicx}
\usepackage{amssymb}
\usepackage{amscd}
\usepackage{hyperref}
\usepackage{verbatim}
\usepackage{amsmath}
\newtheorem{theorem}{Theorem}[section]
\newtheorem{corollary}[theorem]{Corollary}
\newtheorem{lemma}[theorem]{Lemma}
\newtheorem{prop}[theorem]{Proposition}

\theoremstyle{remark}
\newtheorem{rem}[theorem]{\bf Remark}

\theoremstyle{remark}
\newtheorem{qn}[theorem]{\bf Question}

\theoremstyle{definition}
\newtheorem{definition}[theorem]{Definition}

\theoremstyle{remark}

\theoremstyle{definition}
\newtheorem*{ack}{Acknowledgments}

\newcommand{\dbar}{\bar\partial}

\newcommand{\tr}{\, {\rm tr}\,}

\newcommand{\hooklongrightarrow}{\lhook\joinrel\longrightarrow\,}

\newcommand{\I}{\mathfrak{Im}\,}
\newcommand{\R}{\, \mathfrak{Re}\,}

\begin{document}
\title[Unitary representations]{Unitary representations of unimodular Lie groups in Bergman spaces}

\author{Giuseppe Della Sala}
\address{Fakult\"at f\"ur Mathematik\\
Universit\"at Wien\\
Vienna, Austria}
\email{beppe.dellasala@gmail.com}
\thanks{GDS is supported by FWF grant Y377, {\it Biholomorphic Equivalence: Analysis, Algebra and Geometry}}

\author{Joe J Perez}
\address{Fakult\"at f\"ur Mathematik\\
Universit\"at Wien\\
Vienna, Austria}
\email{joe\_j\_perez@yahoo.com}
\thanks{JJP is supported by FWF grant P19667, {\it Mapping Problems in Several
Complex Variables}}

\date{\today}

\maketitle

\begin{abstract} For $G$ an arbitrary unimodular Lie group, we construct strongly continuous unitary representations of $G$ in the Bergman space of a naturally constructed strongly pseudoconvex neighborhood of the complexification of the manifold underlying $G$.\end{abstract}

\section{Introduction}

Let $G$ be a Lie group of real dimension $n$, acting freely by real-analytic transformations on a $m$-dimensional, $C^\omega$ manifold $X$. In \cite{HHK} it is shown that any such $G$ action can be extended to a neighborhood of $X$ in its complexification $X^{\mathbb C} \supset X$, and in this neighborhood, the extended transformations can be chosen to be biholomorphisms. In the same work, the authors also construct a $G$-invariant strongly plurisubharmonic function $\varphi$ which vanishes on $X$, thus by setting
\begin{equation}\label{hhktube}M_\epsilon = \{\varphi <\epsilon\}\subset X^{\mathbb C}\end{equation}
\noindent
for $\epsilon>0$ sufficiently small, one obtains a strongly pseudoconvex $G$-manifold $M_\epsilon$, topologically equal to the Cartesian product of $X$ with a real $n$-dimensional ball, on which $G$ acts freely by holomorphic transformations. The spaces $M_\epsilon$ are called {\it gauged $G$-complexifications of} $X$ in \cite{HHK} and elsewhere are frequently called {\it Grauert tubes}. By construction, the $M_\epsilon$ are Stein manifolds (see also \cite{Gr}) and so possess a rich collection of holomorphic functions $\mathcal O(M_\epsilon)$ which is invariant under the induced group action.

Choosing an invariant measure on $M_\epsilon$ we may analyze the subspace of $\mathcal O(M_\epsilon)$ consisting of those members which are square-integrable. This is called the Bergman space of $M_\epsilon$ and we will denote it $L^2\mathcal O(M_\epsilon)$. The main goal of this paper is to show that, when $G$ is unimodular, the Bergman space of $M_\epsilon$, is large enough to furnish non-trivial unitary representations of $G$. The techniques we use are adaptations of the method of solution of the Levi problem as \cite{K1, FK} in the compact case and which were developed and adapted for noncompact $G$-manifolds in \cite{GHS, P1, P2}.

\noindent
Our main result is

\begin{theorem} If $G$ is a unimodular Lie group of dimension $n$, then there exist a complex manifold $M$, which is topologically the Cartesian product of $G$ with an $n$-ball, and a strongly continuous unitary representation $\mathcal R$ of $G$ in the Bergman space of $M$ such that $\ker \mathcal R$ is a compact subgroup of $G$. \end{theorem}

\begin{corollary} If $G$ has no compact subgroups, then the unitary representations constructed here are faithful.\end{corollary}

\begin{rem} In \cite{GHS}, a $G$-manifold with a nonunimodular structure group is constructed which is Stein but has $L^2\mathcal O = \{0\}$. \end{rem}

Our methods are mainly geometric in the present article as we have previously provided analytic sufficient conditions in \cite{P2} for the existence of a nontrivial, and in fact large, Bergman space on some tubes. Let us describe these methods here briefly.

Mimicking the solution of the Levi problem in \cite{K1, GHS}, $L^2$-holomorphic functions were constructed in \cite{P2} on $G$-manifolds $M$ with compact quotient $\bar M/G$ if the following main assumptions are fulfilled. First, we assume that the group $G$ is unimodular. Second, the complex manifold $M$ is assumed strongly pseudoconvex. Third, roughly speaking, some negative power of a Levi polynomial is required to have the property that convolutions by the group action not smooth its singularities at the boundary of $M$. This last property was dubbed {\it amenability} in \cite{P2} and much of the present article will be concerned with demonstrating that it holds for some tubes.

\section{Preliminaries}

\subsection{Local properties of Levi's polynomial} The main technical material in this paper will concern local details of Levi's polynomial. We will thus go through a fairly thorough description of  this function here. Let $M$ be a complex manifold with nonempty smooth boundary $bM$,
  $\bar M=M\cup bM$, so that $M$ is the interior of $\bar M$,
and ${\rm dim}_{\mathbb C}(M)=n$.  We will also assume for simplicity that $\bar M$ is a closed subset in $\widetilde{M}$, a
complex neighborhood of $\bar M$ so that the complex structure on $\widetilde{M}$ extends that of $M$, and every point of $\bar M$ is an
interior point of $\widetilde{M}$.

Let us choose a smooth function $\rho :\widetilde{M}\to \mathbb R$ so that
\[ M=\{z\mid \rho(z)<0\}, \ \ bM = \{z\mid \rho(z)=0\},\]
\noindent
and for all $x\in bM$, we have $d\rho(x)\neq 0$.
For any $x\in bM$ define the {\it holomorphic tangent plane} to the boundary at $x$ by
\[ T^{\mathbb C}_{x}(bM) = \{w\in \mathbb C^{n}\mid \sum_{k=1}^{n}\left.\frac{\partial\rho}{\partial z_{k}}\right|_x w_{k}=0\}.\]
For $x\in bM$, define the Levi form $L_{x}$ by
\[ L_{x}(w,\bar w) =
  \sum_{j, k=1}^{n} \left.\frac{\partial^{2}\rho}{\partial z_{j}\partial \bar z_{k}}\right|_x  w_{j}\bar w_{k},\quad (w \in T^{\mathbb C}_{z}(bM)).\]
Then $M$ is said to be {\it strongly pseudoconvex} if for every $x\in bM$, the form $L_{x}$ is positive definite. Though $L_x$ depends on $\rho$, its essential features are preserved by biholomorphisms.

Since $\rho$ is real-valued, the Taylor expansion at $x$ of $\rho$ is
\begin{equation}\label{taylor}\rho (z) = \rho (x) + 2 \R f(z,x) + L_x(z-x, \bar z - \bar x) + \mathcal O(|z-x|^3),\quad (z\in\mathbb C^n)\end{equation}
\noindent
with the {\it Levi polynomial} $f$ defined by
\begin{equation}\label{LeviPoly}f(z,x) = \sum_{k=1}^n \left.\frac{\partial \rho}{\partial z_{k}}\right|_x (z_{k} - x_{k}) + \frac{1}{2} \sum_{jk=1}^n  \left.\frac{\partial ^2 \rho}{\partial z_{j} \partial z_{k}}\right|_x(z_{j}-x_{j})(z_{k}-x_{k}).
\end{equation}

As discussed in the introduction, we need to know when convolutions of the singular functions gotten by taking the Levi polynomial to negative powers are not smooth in the boundary. We start with an analysis of $f$ itself. The content of the next two paragraphs is well-known.

In our Grauert tubes, the group invariance and the compactness of the quotient guarantee that without loss of generality (replacing $\rho$ by $e^{\lambda\rho}-1$ with sufficiently large $\lambda>0$) we may choose a defining function of $M$ so that the Levi form $L_x(w,\bar w)$ is positive for all nonzero $w\in\mathbb C^n$ (and not only for $w\in T^c_x(bM)$) and at all points $x\in bM$. 

The complex quadric hypersurface $S_x=\{z\mid f(z,x)=0\}$ has $T_x^{\mathbb C}(bM)$ as its tangent plane at $x$. The strong pseudoconvexity property implies that $\rho(z)>0$ if $f(z,x)=0$ and $z\neq x$ is close to $x$. This means that near $x$ the intersection of $S_x$ with $bM$ contains only $x$. Since $\rho<0$ in $M$, \eqref{taylor} implies that $\R f(z,x)<0$ if $x\in bM$ and $z\in M$ is sufficiently close to $x$. It follows that we can choose a branch of $\log f(z,x)$ so that $z\mapsto\log f(z,x)$ is a holomorphic function in
$z\in M\cap U_x$, where $U_x$ is a sufficiently small neighborhood of $x$ in $M$. Consequently all powers of $f$ are also well-defined and holomorphic in $U_x$. For $\tau<0$, the functions $f^\tau:U_x\to\mathbb C$ are holomorphic in a neighborhood of $x$ and blow up only at $x$.
\begin{lemma}\label{thanksMAS}  Choose local coordinates for which $x\cong 0$ and let
\[a =\left.\frac{\partial \rho}{\partial z_{k}}\right|_0 \quad {\rm and} \quad b = \frac{1}{2} \sum_{jk=1}^n  \left.\frac{\partial ^2 \rho}{\partial z_{j} \partial z_{k}}\right|_0\]
\noindent
so that $f(z,0) = a\cdot z + b z\cdot z$. It follows that for $z$ sufficiently near zero in $\bar M$ there are constants $C,D>0$ so that
\[ C|z|^{2}\le |a\cdot z + b z\cdot z|\le D|z|.\]\end{lemma}
\begin{proof}This is true because
\begin{align} 2|a\cdot z + b z\cdot z|\ge & -2 \R (a\cdot z + b z\cdot z) & \notag \\
\ge & \rho(z)-  2\R (a\cdot z + b z\cdot z) = L_0(z,\bar z)+\mathcal O(|z|^3)\notag \end{align}
\noindent
and the Levi form has a smallest eigenvalue $\lambda>0$, so $L_0(z,\bar z)>\lambda|z|^2$. The other estimate is obvious.\end{proof}
The estimate above yields the following immediately, \cite[Lemma 4.2]{P2}.
\begin{lemma}\label{l2} Let $M$ be a complex manifold, $n=\dim_{\mathbb C}M$, and $\chi\in C^\infty_c(\bar M)$ with small support near zero. Then $\chi f^{-\tau} \in L^p(M)$ whenever $\tau\in(0,n/p)$.\end{lemma}
\subsection{Grauert tubes}\label{ret} As we mentioned in the introduction, if $X$ is a real-analytic manifold on which a Lie group $G$ acts freely by real-analytic transformations, then there exists a complexification $X^{\mathbb C}$ so that $X\hookrightarrow X^{\mathbb C}$ is embedded as a totally real submanifold. It is shown in \cite{HHK} that the action of $G$ on $X$ extends to a neighborhood $V$ of $X$ in $X^{\mathbb C}$ such that the extended transformations are biholomorphisms. Furthermore, there exists a real-valued function $\varphi\in C^\omega(V,\mathbb R)$ with the following properties:

\begin{enumerate}

\item $\varphi$ is constant along the orbits of $G$

\item $\varphi|_X\equiv 0$

\item $\varphi\ge 0$ is strongly plurisubharmonic near $X$.

\end{enumerate}

\noindent
Following \cite{HHK}, define $M_\epsilon = \{\varphi < \epsilon\}\subset V$. Note that $G$ acts freely on $M_\epsilon$ by holomorphic transformations and $M_\epsilon$ is strongly pseudoconvex as its boundary is the level set of a strongly plurisubharmonic function.

In \cite{HHK}, it is also shown that there is an equivariant retraction $R(t,z)$, $t\in [0,1]$, from a neighborhood of $X$ in $M_\epsilon$ onto $X$; in particular the map $R(1,z)$ is a projection $\pi:M_\epsilon\to X$ commuting with the right $G$-action.

We will begin our analysis of $M_\epsilon$ by examining the Taylor series of $\varphi$, {\it cf.}\ VIII, Lemma 1, \cite{HHK}.
\begin{lemma}\label{sqs} Let $M_\epsilon$ and $\varphi$ be as before, and let $p\in M_\epsilon$. There exist local complex coordinates $z_j = x_j + iy_j$, $j=1,\ldots,m$, vanishing at $p$, such that $\varphi(z) = \sum_j y_j^2 + \mathcal O(|z|^3)$.
\end{lemma}
\begin{proof} Since $X\subset M_\epsilon$ is totally real, with a holomorphic change of coordinates, we may assume that $X=\{y_j=0\}$. Since each point of $X$ is a local minimum for $\varphi$, we have $\nabla \varphi|_X\equiv 0$, which implies $\partial^2 \varphi/\partial x_k \partial y_j (0) =0$. Since we also have $\partial^2 \varphi/ \partial x_k\partial x_j(0) = 0$, we obtain
\[\frac{\partial^2 \varphi}{\partial z_j \partial \overline z_k}(0) = \frac{\partial^2 \varphi}{\partial y_j \partial y_k}(0).\]
\noindent
Since $\varphi$ is strongly plurisubharmonic, the form on the left-hand side is positive definite, it follows that the second-order Taylor expansion of $\varphi$ about $0\in \mathbb C^m$ is a positive definite, purely quadratic polynomial involving only the $y_j$. The claim is obtained by diagonalizing the real, symmetric form. \end{proof}

\section{Convolutions of Levi's polynomials}\label{thickening}

From now on we restrict our attention to the case in which $M_\epsilon$ is the gauged $G$-complexification of $G$ acting on itself, {\it i.e.}\ $X=G$, $M_\epsilon\subset G^{\mathbb C}$. We have been unable to demonstrate that $M_\epsilon$ possesses the amenability property mentioned in the introduction (see Definition \ref{amen}). As a result, we {\it thicken} it in the following way.

With $z_0$ a coordinate on $T=(S^1)^{\mathbb C}\cong \mathbb C/ \mathbb Z$, and $z'$ a coordinate in $M_{\delta}\subset V$ from before, define the function $\tilde \varphi(z_0,z')= (\I z_0)^2 + \varphi(z')$ on $T\times M_{\delta}$ with $\delta>0$ sufficiently small. As before, define
\[\tilde M_\epsilon = \{\tilde\varphi<\epsilon\}\]
\noindent
for $0<\epsilon<\delta$. By construction $\tilde\varphi$ is strictly plurisubharmonic, thus $\tilde M_{\epsilon}$ is strongly pseudoconvex for $\epsilon>0$ sufficiently small. Extending the $G$ action to $T\times M_\delta$ by triviality on the $T$ factor, we see that $G$ acts freely on $\tilde M_\epsilon$ by holomorphic transformations, and with compact quotient.
Because the manifold $\tilde M_\epsilon$ admits a free $G$ action,
\begin{equation}\label{hatM}\tilde M_\epsilon = G\times\mathbb T\end{equation}
\noindent
where $\mathbb T$ is a solid torus of $n+2$ real dimensions. Indeed, at each point of the group, we have $\R z_0\in S^1$ free and $\{y = (\I z_0, y_1,\dots,y_n)\mid (\I z_0)^2 + \sum_1^n y_k^2 <\epsilon\}\cong B^{n+1}$ and we have denoted $\mathbb T = S^1\times B^{n+1}$.

We compute the Levi polynomial at the basepoint $(z_0,z') = (i\epsilon, 0)$ of the boundary of $\tilde M_\epsilon$ in the local complex coordinates $z=(z_0, z_1,\ldots, z_n)$.
\begin{lemma}\label{formlev} The Levi polynomial induced by the defining function $\tilde\varphi(z)$ at the point $p=(i\epsilon,0)$ is
\[f(z) = 4 \epsilon -1+4 i z_0 (2 \epsilon -1)+4z_0^2+4 z_1^2+\ldots +4 z_n^2.\]
\end{lemma}
\begin{proof} Follows directly form the form of $\varphi$.\end{proof}
In the following, we are going to denote by $zt$ the action of $t\in G$ on $z\in \tilde M_\epsilon$.
\begin{prop}\label{form} In $\tilde M_{\epsilon}$, consider the curve $(0,\epsilon]\ni s\mapsto (is,0)$ to the basepoint $(i\epsilon, 0)$. There exist coordinates $t=(t_j)_1^n$ in a neighborhood of $e\in G$ such that, for $z$ belonging to the curve, the Levi polynomial at the point $zt$ takes the form
\[f(zt) = \sigma +  t_1^2+\ldots + t_n^2\]
\noindent
and $\sigma = \sigma(s)\to 0$ as $s\to\epsilon$.
\end{prop}
\begin{proof} As $G=X$, the construction in Lemma \ref{sqs} gives real coordinates for $G$ near the identity $e\in G$. We compute $f(zt) = f(z_0,z't)$ along the path obtaining
\[f((is,e)t) = f(is,t) = 4 \epsilon -1+4 i s (2 \epsilon -1)+4s^2+4 t_1^2+\ldots +4 t_n^2\]
\noindent
The $\sigma(s)$ from the claim is simply $4 \epsilon -1+4 i s (2 \epsilon -1)+4s^2$.
\end{proof}
\begin{rem}  What we will use of this proposition is simply the fact that the Levi polynomial on the orbit of this special curve takes the form of a constant plus a norm-like quantity on $G$.\end{rem}
 On our tubes, thickened and unthickened, we have smooth, free, right actions of $G$ with compact quotients. Choosing a biinvariant measure $dt$ on $G$, and fixing smooth measures $dq$ on $B^n$ and $dQ$ on $\mathbb T$, the tensor product measures on the tubes allow us to decompose $L^2(M_\epsilon)$ and $L^2(\tilde M_\epsilon)$ as follows
 \begin{align}\label{decomp} L^{2}(M_\epsilon, dt\otimes dq)\cong & L^{2}(G,dt)\otimes L^{2}(B^n, dq) \notag \\ L^{2}(\tilde M_\epsilon, dt\otimes dQ)\cong & L^{2}(G,dt)\otimes L^{2}(\mathbb T, dQ),\end{align}
\noindent
which also present the Hilbert spaces as free Hilbert $G$-modules, \cite{GHS}. Later, $dq$ and $dQ$ will be chosen in a coherent way, but for now we think of them as being arbitrary smooth measures.

On these manifolds, the global right $G$-action and the Haar measure $dt$ combine to define convolution operators in $L^2$ which we will write
\[(R_\Delta u)(z) = \int_G dt\,  \Delta(t) u(zt),\]
for example, for $\Delta\in L^1(G)$, $u\in L^2$.

Let us begin our analysis of the asymptotics of convolutions of powers of $f$. For $U_x$ a coordinate neighborhood at the boundary of $\tilde M_\epsilon$ choose a cut-off function $\chi\in C_c^\infty(U_x)$, so that $\chi=1$ in a neighborhood of $x$.
\begin{lemma}\label{prev}\cite[Rem.\ 4.3]{P2} Let $M$ be a strongly pseudoconvex $G$-manifold. Assume for simplicity that $M$ is a trivial $G$-bundle and let $f$ be a Levi polynomial on $M$. Then for all $x\in\bar M/G$, $\|\chi f^{-\tau}(\cdot,x)\|_{L^1(G)}<\infty$ as long as $2\tau<\dim_{\mathbb R}G$.\end{lemma}
We will also need the following fact regarding convolutions of powers of Levi's polynomial.
\begin{lemma}\label{L2} Let $n=\dim_{\mathbb R}G$, $\tau\in(0,n/2)$ and $\Delta\in L^2(G)$. Then $R_\Delta\chi f^{-\tau}\in L^2(M)$. \end{lemma}
\begin{proof} Since $meas(M/G)<\infty$, Young's and Jensen's inequalities give
\begin{align}\label{Jen}  \|R_\Delta h\|_{L^2(M)}^2  =& \int_{M/G} dx\int_G dt\ \left|\int_Gds\ \Delta(s)h(ts,x) \right|^2\\
& \le \|\Delta\|_{L^2(G)}^2\int_{M/G} dx\ \|h(\cdot,x) \|_{L^1(G)}^2 \notag\\
& \lesssim \|\Delta\|_{L^2(G)}^2\left|\int_{M/G} dx\ \|h(\cdot,x)\|_{L^1(G)}\right|^2 \notag \\
& =  \|\Delta\|_{L^2(G)}^2\|h\|_{L^1(M)}^2.\notag \end{align}
\noindent
Lemmata \ref{l2} and \ref{prev} provide that $h=\chi f^{-\tau}\in L^1(M)$, which gives the result. \end{proof}
\begin{rem}\label{comp} Note that in the expression \eqref{Jen}, with $h=\chi f^{-\tau}$, the $G$-integral is over a compact neighborhood of the identity. This neighborhood can be chosen as small as we like, choosing $\chi$ accordingly.\end{rem}

As we have mentioned in the introduction, in \cite{P2} a sufficient condition is given guaranteeing that the Bergman space of any strongly pseudoconvex $G$-manifold with compact quotient be large in the sense of von Neumann's $G$-dimension. It is the following
\begin{definition}\label{amen} Let $\xi:\bar Y\to M$ be a piecewise continuous section of $G\to M\stackrel p \to Y$ so that $\xi|_{p({\rm supp}\chi)}$ is continuous. The action of $G$ on $M$ is called \emph{amenable} if there exist an $x\in bM$ and $\tau>0$ so that if $f$ is a Levi polynomial at $x$, then 1) $\chi f^{-\tau}\in L^2(M)$, 2) $\|\chi f^{-\tau}(\cdot,\xi)\|_{L^1(G)}<\infty$ for all $\xi\in \bar Y$, and 3) $R_\Delta\chi f^{-\tau}\notin C^\infty(\bar M)$ for all nonzero $\Delta\in C^\infty(G)$.\end{definition}
\begin{rem} The reader will note that the formulation of the definition seems to be dependent on the defining function.\end{rem}
\begin{prop}\label{gactionamen} The $G$ action on the manifold $\tilde M_\epsilon$ is amenable.\end{prop}
\begin{proof} First note that if $\dim_{\mathbb R}G = n$, then $\dim_{\mathbb C}\tilde M_\epsilon = n+1$, thus
\begin{enumerate}
\item $\chi f^{-\tau}\in L^2(M)$ if $\tau\in [0,\frac{n+1}{2})$
\item $\|\chi f^{-\tau}(\cdot,\xi)\|_{L^1(G)}<\infty$ for all $\xi\in X$ if $2\tau<\dim_{\mathbb R}G = n$.
\end{enumerate}
\noindent
by Lemmata \ref{l2} and \ref{prev}, respectively. So let us choose $\tau <n/2$ and estimate the convolution along a path to the basepoint in $b\tilde M_\epsilon$, using Prop.\ \ref{form}:
\begin{align}\label{convolvomit}(R_\Delta f^{-\tau})(is/2,0)\sim& \int dt\ \left[   4 \epsilon -1- 2(2 \epsilon -1) s -s^2+4 |t|^2\right]^{-\tau}\notag\\
\sim & \int_0^\delta\frac{r^{n-1} dr }{[4 \epsilon -1- 2(2 \epsilon -1) s -s^2+4 r^2]^\tau}\notag\\
\sim& \int_0^\delta \frac{r^{n-1} dr}{[\sigma+4 r^2]^\tau}\notag\end{align}
\noindent
where $\sigma\to 0$ as $s\to 1$, and where we have invoked the fact in Rem.\ \ref{comp}. Away from the path of the singularity, the convolution is smooth, thus we have
\[\lim_{s\to 1}\frac{\partial^k}{\partial s^k} (R_\Delta f^{-\tau})(is/2,0) \longrightarrow \infty\]
\noindent
for $\tau>0$, $\tau+k>n$. Taking $\tau=1$ suffices.\end{proof}
\begin{rem} Estimates of convolutions in substantially greater generality suggest that the amenability property is satisfied whenever the group action avoids the ``bad'' direction. For example, in the language of \cite{FS}, subgroups of the Heisenberg group of the form $\{(z,0)\mid z\in\mathbb C^n\}$ lead to amenable actions while those containing $\{(0,t)\mid t\in\mathbb R\}$ do not. This tempts us to conjecture that in the gauged $G$-complexification $M$ of a $G$-manifold $X$ for which $\dim_{\mathbb R} X > \dim_{\mathbb R} G$, the extended $G$-action on $M$ is amenable.\end{rem}
\begin{theorem}\label{Gdim} $\dim_G L^2\mathcal O(\tilde M_{\epsilon}) = \infty$\end{theorem}
\begin{proof} $\tilde M_{\epsilon}$ satisfies the requirements of Thm.\ 5.2 in \cite{P2}.\end{proof}
\begin{rem} The positivity of the $G$-dimension is sufficient to obtain the infinite-dimensionality of the Bergman space of $\tilde M_\epsilon$ over the complex numbers.\end{rem}

\section{Bergman representation spaces}

\subsection{Preparatory geometric considerations} Any $G$-invariant Riemannian metric on $\bar M$ is complete in the following sense. For any point $x_0\in\bar M$ and for any $R>0$ the ball $B(t, R) = \{s\in\bar M : {\rm dist}(t,s) < R\}$ of the corresponding geodesic metric is relatively compact in $\bar M$, \cite{GHS}.

We will need the following topological lemma. For $K\subset G$ and $t\in G$, denote by $K t$ the set $\{kt: k\in K\}\subset G$.
\begin{lemma}\label{lK} Let $G$ be a non-compact Lie group, $K\subset G$ a compact subset containing the identity, and $L\subset G$ an unbounded sequence. It follows that there exists a $t\in L$ such that $K\cap K t = \emptyset$.\end{lemma}
\begin{proof} Choose a right-invariant Riemannian metric $d$ on the group $G$. By the observation above, $G$'s closed and $d$-bounded subsets are compact. For $K$ as in the statement, choose $R>0$ such that $K\subset B(e,R)$. Suppose, by contradiction, that $K\cap K t\neq \emptyset$ for all $t\in L$, and choose $k_t\in K$ such that $k_t t\in K$. Then for any $p\in Kt$ ($p=p't$ with $p'\in K$) we have
\[d(e,p)\leq  d(e,k_t t) + d(k_t t,p) = d(e,k_t t) + d(k_t,p')\]
\noindent
since, by invariance of $d$, we have $ d(k_t t,p)  = d(k_t,p')$. We continue, noting that
\[\dots\leq d(e,k_t t) + d(k_t,e) + d(e,p')\leq 3R,\]
\noindent
thus $K t\subset B(e,3R)$ for all $t\in L$. It follows that $K L = \{k t: t\in L, k\in K\}\subset B(e,3R)$, so that $L\subset KL$ is bounded and hence relatively compact, a contradiction.\end{proof}

\subsection{Restrictions to the Grauert tubes of $G$}\label{rtgt}

Denoting the real $n$-ball by $B^n$, the Grauert tubes of $X=G$ of the form $M_\epsilon$ from \eqref{hhktube} embed in $\tilde M_\delta$ in the following way:
\begin{equation}\label{inj}M_\epsilon = G\times B^n \hookrightarrow G\times\mathbb T = \tilde M_\delta.\end{equation}
\noindent
More specifically, in the notation of Sect.\ \ref{thickening}, for every fixed $z_0 = x_0+iy_0$, the map $ M_{\delta - y_0^2} \ni p \to (x_0,y_0,p) \in \tilde M_\delta$ is an embedding of $M_{\delta - y_0^2}$ as a complex submanifold of $\tilde M_\delta$ with ${\rm codim}_{\mathbb C}\, M_\delta = 1$. Thus the restriction to the copy of $M_\epsilon$ at $z_0$, denoted $M_{\epsilon}^{z_0}$, is a holomorphic map
\[\mathcal O (\tilde M_{\delta})\ni f \longmapsto f(z_0, \cdot)\in\mathcal O (M_{\epsilon}^{z_0})\]
and so elements of $L^2\mathcal O(\tilde M_\delta)$ restrict to elements of $\mathcal O M_{\epsilon}^{z_0}$ as long as $y_0^2 +\epsilon \le \delta$. 

The holomorphic functions constructed in Thm.\ \ref{Gdim} are, by construction, smooth in the closure of $\tilde M_\delta$ except along the orbit of the basepoint of $f$. Thus the restrictions of these functions are smooth in the closures of the $M_\epsilon^{z_0}$.

In the decompositions \eqref{decomp}, choose the measure $Q$ on $\mathbb T = \tilde M_\delta/G$ as a tensor product $dq \otimes d\lambda$ for a suitable measure $d\lambda$ on the cylinder $T$ and $dq$ on the slices $\{z_0 = {\rm const.}\}$\ . 

We will denote by $t_*h$ the function
\[t_* h: z\longmapsto h(zt)\]
which is in $L^2\mathcal O(M_\epsilon)$ by the $G$-invariance of the measure on $M_\epsilon$.

\begin{lemma}\label{sequence} Let $h\in L^2(M_\epsilon)$ with $\|h\|_{L^2} = 1$. Let $L\subset G$ be a sequence, chosen in such a way that no subsequence $(t_k)_k\subset L$ has compact closure. It follows that the complex vector space $\langle h\rangle_L$ generated by $\{t_* h: t\in L\}$ is infinite-dimensional over $\mathbb C$. \end{lemma}
\begin{proof} Suppose, by contradiction, that $\{t_* h\mid t\in L\}$ is contained in a finite-dimensional space. Since $\|t_* h\|_{L^2}=1$ for all $t\in L$, there exists a sequence $(t_k)_k\subset L$ such that $(t_{k*} h)_k$ is convergent. Denote by $h_0\in L^2(M_\epsilon)$ the limit of this sequence and note that $\|h_0\|_{L^2}=1$.

Let $K\subset G$ be a fixed, arbitrary compact set containing the identity. We will obtain a contradiction to the assertion that $\|h_0\|_{L^2} = 1$ by showing that $h_0|_{\pi^{-1}(K)}\equiv 0$, where $\pi$ is the $G$-equivariant projection map $\pi:M_\epsilon\to G$, {\it cf.}\ Sect.\ \ref{ret}.

Fix $\epsilon_0>0$, and let $K_0\supset K$ be a compact set such that the function $h$ from the statement satisfies
$$
\int_{M_\epsilon \setminus \pi^{-1} (K_0)}|h|^2 dt\otimes dq < \epsilon_0.
$$
By a repeated application of Lemma \ref{lK}, we obtain a sequence $(t_k)_k$ in $L$ such that $K_0 \cap K_0 t_k = \emptyset$, thus
$$
\pi^{-1}(K_0)\cap \left[\pi^{-1}(K_0) t_k\right] = \pi^{-1}(K_0 \cap K_0 t_k) = \emptyset
$$
for all $k$, by the equivariance of $\pi$. It follows that
\[ \int_{\pi^{-1}(K_0)}|{t_{k*}} h|^2 
 = \int_{\pi^{-1}(K_0)t_k}|h|^2   \leq  \int_{M_\epsilon \setminus \pi^{-1}(K_0)} |h|^2  < \epsilon_0.\]
Since $K\subset K_0$ we have $\int_{\pi^{-1}(K)}|{t_{k*}} h|^2  < \epsilon_0$ and so $\|h_0\|_{L^2(\pi^{-1}(K))}\le\epsilon_0$ since $h_0$ is the limit of $(t_{k*}h)_k$. By letting $\epsilon_0 \to 0$ we obtain $h_0|_{\pi^{-1}(K)}\equiv 0$.\end{proof}
\begin{prop}\label{infdim} Let $M_\epsilon$ be the Grauert tube of $G$. It follows that the space $L^2\mathcal O(M_\epsilon)$ is infinite-dimensional as a vector space over $\mathbb C$. \end{prop}
\begin{proof} Let $f\in L^2\mathcal O(\tilde M_\delta)$ be a non-trivial function provided by Theorem \ref{Gdim}. From Fubini's theorem follows
$$\infty > \int_{\tilde M_\delta} |f(z_0,z)|^2 dt\otimes dQ = \int_{\tilde M_\delta} |f(z_0,z)|^2 dt\otimes dq \otimes d\lambda = $$
$$=\int_{\{y_0^2<\delta\}}d\lambda(z_0) \int_{M^{z_0}_\epsilon} |f(z_0,z)|^2 (dt\otimes dq)(z)
$$
\noindent
where at each height $y_0$ we have chosen the radius $\epsilon = \epsilon(y_0) = \delta - y_0^2$. Thus the integral of $|f(c,z)|^2$ over $\tilde M_\epsilon \cap \{z_0=c\} \cong M_{\epsilon - (\I c)^2}$ is finite for almost every $c\in \{(\I z_0)^2 < \epsilon\}$; in other words, for $\epsilon'<\epsilon$ in a set of positive measure, we obtain that $L^2\mathcal O(M_{\epsilon'})\neq \{0\}$.

Let, then, $h$ be a non-trivial element of $L^2 \mathcal O(M_\epsilon)$. Since $G$ is non-compact, we can choose a sequence $L\subset G$ as in Lemma \ref{sequence} (for example, we can fix an exhaustion $K_j$ of $G$ by compact subsets and choose $t_j\in K_j\setminus K_{j-1}$). Then $\langle h\rangle_L$ is still contained in $L^2\mathcal O(M_\epsilon)$; by Lemma \ref{sequence} follows that $\dim_{\mathbb C} L^2\mathcal O(M_\epsilon) = \infty$.\end{proof}

In the subspace of the Bergman space that we have constructed, we have a faithful unitary representation of $G$  \lq\lq modulo compact subgroups.\rq\rq\ That is the content of the following
\begin{theorem} Let $G$ be a unimodular group. Then, there is a representation $\mathcal R$ of $G$ in $L^2\mathcal O(M_\epsilon)$, where $M_\epsilon$ is a neighborhood of $G$ in its complexification, such that $\ker \mathcal R$ is a compact subgroup of $G$. In particular, if $G$ does not have compact subgroups then $\mathcal R$ is a faithful representation.
\end{theorem}
\begin{proof} Define $\mathcal R$ to be the natural representation of $G$ in the space of unitary operators of $L^2\mathcal O(M_\epsilon)$, with $M_\epsilon$ as in the previous section, induced by $\mathcal R(t)f(z) = t_* f(z) = f(zt)$ for any $t\in G$ and $f\in L^2\mathcal O(M_\epsilon)$.

Let $H$ be the kernel of $\mathcal R$, and choose a function $f\in L^2\mathcal O(M_\epsilon)$, $f\neq 0$; then $t_* f = f$ for all $t\in H$. In particular, the space generated by the translates of $f$ by elements of $H$ is $1$-dimensional. By the same arguments as in Cor.\ \ref{infdim}, $H$ must be relatively compact because otherwise, we could find a discrete subset $L\subset H$ with the property that every infinite subsequence of $L$ is not relatively compact.  \end{proof}

\begin{rem} The representations constructed are strongly continuous.\end{rem}

\section{The thickened Heisenberg group $\mathbb H_3(\mathbb R)\hookrightarrow \mathbb H_3(\mathbb C)$}

We will describe here a simple, detailed example of the manifolds in question. For $\mathbb K =\mathbb Z,\, \mathbb R,$ or $\mathbb C$, define
{\tiny \begin{equation}\label{vars}\mathbb H_3(\mathbb K)=\left\{\left(\begin{array}{ccc}
1 & z_1 &  z_3 \\
0 & 1  & z_2 \\
0 & 0 & 1 \end{array}\right) \mid z_k\in\mathbb K \right\} \quad {\rm and} \quad \mathfrak{h}_3(\mathbb R) = \left\{\left(\begin{array}{ccc}
0 & \theta_1 &  \theta_3 \\
0 & 0  & \theta_2 \\
0 & 0 & 0 \end{array}\right)\mid \theta_k\in\mathbb R \right\}.\end{equation}}
For $\mathbb H_3(\mathbb R)\hookrightarrow \mathbb H_3(\mathbb C)$, close to the origin of $\mathbb H_3(\mathbb C)$, the {\it slice through} $e$, $S_e\subset \mathbb H_3(\mathbb C)$, is the exponential of the normal vectors of the identity in the natural inclusion:
\begin{equation}\label{decompheis}T_e\mathbb H_3(\mathbb R)\oplus\mathbb H_3(\mathbb R)^\perp \hooklongrightarrow T_e \mathbb H_3(\mathbb C); \quad S_e=\exp T_e\mathbb H_3(\mathbb R)^\perp\subset \mathbb H_3(\mathbb C). \end{equation}
{\it I.e.}\ $S\subset\mathbb H_3(\mathbb C)$ consists of matrices of the form $\exp[i \Theta]$ with $\Theta = (\theta_{jk})_{jk}$ in $\mathfrak{h}_3(\mathbb R)$. In \cite{HHK}, an invariant strongly plurisubharmonic function $\varphi$ is constructed abstractly; here, the goal is to describe $\varphi$.

Expressing an arbitrary element $Z$ in the form $Z=\exp[i\Theta] t$, with $t\in \mathbb H_3(\mathbb R)$ and $\Theta\in\mathfrak{h}_3(\mathbb R)$, as in the description \eqref{decompheis}, the definition of $\varphi$ given in \cite{HHK} becomes
\begin{equation}\label{sum}\varphi(\exp[i\Theta]t) = \sum_{jk} \theta_{jk}^2.\end{equation}
\noindent
Note that this definition makes the right-$\mathbb H_3(\mathbb R)$ invariance of $\varphi$ manifest. So, given $Z\in\mathbb H_3(\mathbb C)$, consider the factorization
\begin{equation}\label{stroke}Z = \exp[i\Theta]t\end{equation}
\noindent
with $t\in\mathbb H_3(\mathbb R)$ and $\Theta$ a real matrix. Defining $X=\R Z$ and $Y=\I Z$, the identity $\exp[i\Theta] = \cos[\Theta] + i \sin[\Theta]$, provides
\[ Z = \R Z + i\I Z = X + i Y = (\cos[\Theta] + i \sin[\Theta])t.\]
\noindent
Equating real and imaginary parts,
\[X= \cos[\Theta]t, \qquad Y=\sin[\Theta]t,\]
\noindent
thus
\[t =[\cos\Theta]^{-1}X  = [\sin\Theta]^{-1}Y\]
\noindent
and so
\begin{equation}\label{gen}\Theta=\tan^{-1}(YX^{-1}) = YX^{-1}-\frac{(YX^{-1})^3}{3}+\frac{(YX^{-1})^5}{5}-\dots.\end{equation}
\noindent
We easily compute $\exp[i\Theta]$ for $\Theta\in\mathfrak{h}_3(\mathbb R)$;
{\tiny \[\exp\left[i\left(\begin{array}{ccc}
0 & \theta_1 &  \theta_3 \\
0 & 0  & \theta_2 \\
0 & 0 & 0 \end{array}\right)\right] = \left(\begin{array}{ccc}
1 & i\theta_1 &  i\theta_3 - \theta_1\theta_2/2 \\
0 & 1  & i\theta_2 \\
0 & 0 & 1 \end{array}\right), \]}
\noindent
thus the factorization \eqref{stroke} is explicitly
\begin{align}\label{fac}Zt=& \exp[i\Theta]t = Xt+iYt  \\
&{\tiny =\left(\begin{array}{ccc}
1 & i\theta_1 &  i\theta_3 - (1/2)\theta_1\theta_2 \\
0 & 1  & i\theta_2 \\
0 & 0 & 1 \end{array}\right) \left(\begin{array}{ccc}
1 & t_1 &  t_3  \\
0 & 1  & t_2 \\
0 & 0 & 1 \end{array}\right)}\notag \end{align}
\noindent
and any element of $\mathbb H_3(\mathbb C)$ has a unique decomposition of this form. For a general $Z\in\mathbb H_3(\mathbb C)$, computing $YX^{-1}$, which we need in \eqref{gen}, we get
\begin{equation}\label{yx-1}{\tiny YX^{-1}=\left(\begin{array}{ccc}
0 & y_1 & y_3 -x_2 y_1 \\
0 & 0  & y_2 \\
0 & 0 & 0 \end{array}\right),}\end{equation}
\noindent
where $z_k = x_k + iy_k$. The next step is to compute the arctangent, which is again easy because of the nilpotence. We obtain
\[\Theta = \tan^{-1}(YX^{-1}) = YX^{-1} .\]
\noindent
Denoting by $\Theta^T$ the transpose of $\Theta$, the sum of the squares of the elements of $\Theta$ as in Eq.\ \eqref{sum} can be written $\tr(\Theta^T \Theta)$ in general, but for our purposes, it is enough to read off
\begin{align}\label{spsh}\varphi(Z) =& \tr(\tan^{-1}(YX^{-1})^T\tan^{-1}(YX^{-1})) = \tr((YX^{-1})^TYX^{-1})\notag \\
=&  (\I z_1)^2 + (\I z_2)^2  +(\I z_3 - \R z_2\, \I z_1)^2,\end{align}
\noindent
from \eqref{yx-1}. This is an invariant, strongly plurisubharmonic function in a neighborhood of $\mathbb H_3(\mathbb R)$ in $\mathbb H_3(\mathbb C)$. Now, as in Sect.\ \ref{thickening}, consider $(S^1\times \mathbb H_3(\mathbb R))^{\mathbb C} \cong \mathbb C/\mathbb Z \times \mathbb H_3(\mathbb C)$ with $\mathbb H_3(\mathbb R)$ acting from the right, and trivially on the first factor:
\[\mathbb H_3(\mathbb R) \ni t : (z_0; Z)\longmapsto (z_0, Zt).\]
\noindent
Defining a new function by $\tilde\varphi(z_0, Z) = (\I z_0)^2 + \varphi(Z)$, an easy calculation shows that $\tilde M_\epsilon = \{\tilde\varphi <\epsilon\}$ is strongly pseudoconvex as long as $\epsilon<1$. With the definition \eqref{LeviPoly}, at $z^0=(z_0, z_1,z_2,z_3) = (i\epsilon/2, 0,0,0)\in\{\tilde\varphi=\epsilon^2/4\}$, $\tilde\varphi$ induces a Levi polynomial proportional to
\[f(z):= 4 \epsilon -1+4 i z_0 (2 \epsilon -1)+4z_0^2+4 z_1^2+4 z_2^2+4 z_3^2,\]
\noindent
as in Lemma\ \ref{formlev}. The action of $t\in \mathbb H_3(\mathbb R)$ on $f$ is given by
\[f(zt) =  4 \epsilon -1+4 i z_0 (2 \epsilon -1)+4z_0^2+4 (z_1+t_1)^2+4 (z_2+t_2)^2+4 (z_3+ z_1 t_2 + t_3)^2\]
\noindent
The convolution of $\chi f^{-1}$ by a convolution kernel $\Delta\in C^\infty(\mathbb H_3(\mathbb R))$ is approximated in a neighborhood of the basepoint $z^0$ by 
{\tiny\[  (R_\Delta \chi f^{-1})(z)\sim \int dt\ \left[   4 \epsilon -1+4 i z_0 (2 \epsilon -1)+4z_0^2+4 (z_1+t_1)^2+4 (z_2+t_2)^2+4 (z_3+ z_1 t_2 + t_3)^2\right]^{-1}\]}
\noindent
where we have shortened $dt_1 dt_2 dt_3$ to simply $dt$ and with all the integrals over small intervals $t_k\in (-\epsilon_k,\epsilon_k)$, $k=1,2,3$. Now consider the path from inside the manifold to the basepoint $(i\epsilon/2,0,0,0)$ given by $s\mapsto(is/2,0,0,0)$, $s\to \epsilon^-$. Along this path, the convolution simplifies to
\begin{align}\label{convolvomit}(R_\Delta f^{-1})(i\epsilon/2,0,0,0)\sim& \int dt\ \left[   4 \epsilon -1- 2(2 \epsilon -1) s -s^2+4 t_1^2+4 t_2^2+4t_3^2\right]^{-1}\notag \\
\sim& \int_0^\delta \frac{r^2 dr}{[\sigma+4 r^2]^\tau}\end{align}
\noindent
with $\sigma\to 0$ as $s\to \epsilon^-$, as in Prop.\ \ref{gactionamen}.

Note that since $\mathbb H_3(\mathbb R)$ has a cocompact discrete subgroup, $\mathbb H_3(\mathbb Z)$, the construction of $L^2$ holomorphic functions in \cite{GHS} yields other elements of $L^2\mathcal O(\{\varphi<\epsilon\})$, these having local peak points in the boundary and with no thickening. Of course, in such a concrete example, one could also construct elements of $L^2\mathcal O$ by hand. Also note that a generic unimodular Lie group possesses no discrete cocompact subgroups, \cite{M}.

We end this example with a question: Is it possible to follow through with the restriction in Sect.\ \ref{rtgt} to determine the representation actually obtained here and then relate that to the representations of $\mathbb H_3(\mathbb R)$ in \cite{T}? 

\begin{ack} The authors thank Frank Kutzschebauch and Bernhard Lamel for helpful conversations and the Erwin Schr\"odinger Institute for its generous hospitality. \end{ack}


\end{document}